\documentclass[11pt,regno]{amsart}
\usepackage{amsmath,amsthm,amsfonts,amssymb,amscd,euscript,graphicx,overpic}

\newtheorem{theorem}{Theorem}[section]
\newtheorem{maintheorem}{Theorem}
\newtheorem{lemma}[theorem]{Lemma}
\newtheorem{corollary}[theorem]{Corollary}
\newtheorem{proposition}[theorem]{Proposition}
\theoremstyle{definition}

\newtheorem{remark}[theorem]{Remark}
\newtheorem*{claim*}{Claim}
\newtheorem*{remark*}{Remark}

\numberwithin{equation}{section}

\newcommand{\eqdef}{\stackrel{\scriptscriptstyle\rm def}{=}}

\def\sst{{ss}}

\def\uut{{uu}}

\def\bN{\mathbb{N}}
\def\bZ{\mathbb{Z}}
\def\bP{\mathbb{P}}
\def\bH{\mathbb{H}}
\def\bR{\mathbb{R}}

\def\cC{\EuScript{C}}

\def\cM{\EuScript{M}}
\def\cE{\EuScript{E}}

\def\cT{\EuScript{T}}

\def\ZZ {{\mathbb Z}}\def\NN {{\mathbb N}}\def\CC {{\mathbf C}}

\DeclareMathSymbol{\varnothing}{\mathord}{AMSb}{"3F}
\renewcommand{\emptyset}{\varnothing}
\title[Abundant rich phase transitions]{Abundant rich phase transitions\\ in step skew products}
\author[L.~J.~D\'iaz]{L. J. D\'\i az}
\address{Departamento de Matem\'atica PUC-Rio, Marqu\^es de S\~ao Vicente 225, G\'avea, Rio de Janeiro 225453-900, Brazil}
\email{lodiaz@mat.puc-rio.br}
\author[K.~Gelfert]{K.~Gelfert}
\address{Instituto de Matem\'atica Universidade Federal do Rio de Janeiro, Av. Athos da Silveira Ramos 149, Cidade Universit\'aria - Ilha do Fund\~ao, Rio de Janeiro 21945-909,  Brazil}\email{gelfert@im.ufrj.br}
\author[M.~Rams]{M. Rams} \address{Institute of Mathematics, Polish Academy of Sciences, ul. \'{S}niadeckich 8, 00-956 Warszawa, Poland}
\email{m.rams@impan.gov.pl}
\thanks{This paper was partially supported by CNPq, Faperj, and Pronex (Brazil). The authors thank the hospitality of IM UFRJ and PUC Rio.
}

\begin{document}

\begin{abstract}
	We study phase transitions for the topological pressure of geometric potentials of transitive sets. The sets considered are partially hyperbolic having a step skew product dynamics over a horseshoe with one-dimensional fibers corresponding to the central direction. The sets are genuinely non-hyperbolic  containing intermingled horseshoes of different hyperbolic behavior (contracting and expanding center). 

We prove that for every $k\ge 1$ there is a diffeomorphism $F$ with a transitive set $\Lambda$ as above such that the pressure map $P(t)=P(t\, \varphi)$ of the potential $\varphi= -\log \,\lVert dF|_{E^c}\rVert$ ($E^c$ the central direction) defined on $\Lambda$ has $k$ rich phase transitions. This means that there are parameters $t_\ell$, $\ell=1,\ldots,k$, where $P(t)$ is not differentiable and this lack of differentiability is due to the coexistence of two equilibrium states of $t_\ell\,\varphi$ with positive entropy and different Birkhoff averages.
Each phase transition is associated to a gap in the central Lyapunov spectrum of $F$ on $\Lambda$.
\end{abstract}


\maketitle

\section{Introduction}\label{sec:1}

Thermodynamical formalism is a branch of ergodic theory, that studies quantifiers of invariant measures such as entropy and Lyapunov exponents.
 Those are interrelated by the pressure of the so-called \emph{geometric potentials} associated to the differential of the map.
 Recall that, given a continuous map $F\colon \Lambda\to \Lambda$ of a compact metric space $\Lambda$ and a continuous function $\varphi\colon\Lambda\to\bR$, its topological pressure $P_{F|\Lambda}(\varphi)$ with respect to $F|_\Lambda$ is defined in purely topological terms  and can be expressed via the variational principle
 \[
 	P_{F|\Lambda}(\varphi)=\sup\Big(h(\mu)+\int\varphi\,d\mu\Big),
 \]
 where $h(\mu)$ denotes the entropy of the measure $\mu$ and the supremum can be taken over all ergodic $F$-invariant probability measures (see~\cite{Wal:81}). A measure is called an \emph{equilibrium measure for $\varphi$} with respect to $F|_\Lambda$ if it attains the supremum in the above principle.

 This theory is quite well understood in the uniformly hyperbolic setting, see~\cite{Bow:08,Rue:78}. In this setting, for a given  H\"older continuous function $\varphi$ (sometimes called potential) the pressure function  $t\mapsto P_{F|\Lambda}(t\,\varphi)$, $t\in\bR$, is analytic and its derivatives are related to stochastic properties of the dynamics and of equilibrium states. Little is known beyond this hyperbolic context and naturally, a  crucial point is the differentiability of the pressure function.

Going beyond the hyperbolic scenery, in the non-hyperbolic context there are essentially two types of dynamics called \emph{critical} and \emph{non-critical} (see Preface in~\cite{BonDiaVia:05}). The former one is present for instance in the quadratic family and for diffeomorphisms with homoclinic tangencies, while the latter one is mainly associated to partially hyperbolic dynamics.
In this paper we consider partially hyperbolic transitive sets of a diffeomorphism $F$ with one-dimensional center direction $E^c$ and study the differentiability of the pressure for the geometric potential $\varphi=-\log\,\lVert dF|_{E^ c}\rVert$.
Let us observe that this map is always convex and thus non-differentiable in at most a countable number of points (see~\cite[Chapter 9]{Wal:81} for more information).
Moreover, in this setting equilibrium states for $t\,\varphi$ always exist for every $t\in\bR$ (see~\cite[Proposition 6]{CowYou:05} and also~\cite{DiaFis:11}).

Before stating the result of the paper, let us recall that the pressure function $t\mapsto P_{F|\Lambda}(t\,\varphi)$, $t\in\bR$, is said to exhibit a \emph{phase transition} at a parameter $t_c$ if it fails to be real analytic at $t_c$. This transition is of \emph{first order} if it fails to be differentiable at $t_c$. In some cases, the existence of a phase transition is associated to the existence of (at least) two (ergodic) equilibrium states with different Birkhoff averages. Note that to each ergodic measure there is associated a sub-gradient of the pressure function whose slope is given by the Birkhoff average of $\varphi$ with respect to the measure. The transition is said to be \emph{rich} if there are at least two equilibrium states for $t_c\,\varphi$ with positive entropy.

\begin{maintheorem}\label{t1}
	Given any $k\in\bN$ and any manifold $M$ of dimension at least $3$, there is a $C^\infty$ diffeomorphism $F$ of $M$ having a compact invariant set $\Lambda$ that is topologically transitive and partially hyperbolic with splitting
	\[T_\Lambda M=E^s\oplus E^c\oplus E^u\]
with non-trivial bundles and one-dimensional center bundle $E^c$ having the following properties:
	
	Consider the continuous potential $\varphi= -\log\,\lVert dF|_{E^c}\rVert$, the topological pressure $P(t)=P_{F|\Lambda}(t\,\varphi)$ has $k$ rich phase transitions. More precisely, there are parameters $t_k<\cdots<t_1<0$ such that $P(t)$ is real analytic in each interval $(-\infty,t_k), (t_k,t_{k-1}),\ldots,(t_2,t_1)$ and not differentiable at $t_k,\ldots,t_1$. Further, for every $t_\ell$, $\ell=1,\ldots,k$, there exist at least two equilibrium states for $t_\ell\,\varphi$ with respect to $F|_\Lambda$ both with positive entropy.
\end{maintheorem}

The dynamics of the transitive set $\Lambda$ is a step skew product with one-dimensional fibers, semi-conjugate to a shift map over three symbols, and its central direction has mixed contracting and expanding behaviors. For instance, the sets of those hyperbolic periodic points which are contracting in the central direction and those which are expanding in the central direction are both dense in $\Lambda$. This sort of sets have also a rich fiber structure and contain uncountably many curves
(called spines) tangent to the central direction.
Topological properties and the rich fibre structure of such sets, called porcupine horseshoes, were studied in~\cite{DiaHorRioSam:09, DiaGel:12,DiaGelRam:}, compare also so-called bony attractors with a somehow similar structure in~\cite{Ily:10,Kud:10}.  The occurrence of phase transitions for the geometric potential were studied in~\cite{LepOliRio:11,DiaGel:12} (existence of \emph{one} transition) and \cite{DiaGelRam:11} (existence of \emph{one rich} transition).
\smallskip

Let us briefly describe the diffeomorphism $F$ in a neighborhood of $\Lambda$. Consider the cube $\widehat \CC=[0,1]^2$ and a diffeomorphism $\Phi$ defined on $\bR^2$ having a horseshoe $\Gamma$ in $\widehat\CC$ conjugate to the full shift $\sigma\colon\Sigma_{012}\to\Sigma_{012}$ of three symbols, $\Sigma_{012}=\{0,1,2\}^\bZ$. Denote by  $\varpi\colon \Gamma \to \Sigma_{012}$ the conjugation map $\varpi \circ \Phi=\sigma \circ \varpi$.
Let $\CC= \widehat \CC\times [0,1]$. We consider a map $F\colon\CC\to \bR^3$ defined by
\begin{equation}\label{e.defF}
    F(\widehat x,x) \eqdef
        (\Phi(\widehat x), f_{\varpi(\widehat x)}(x))
\end{equation}
where $f_{\varpi(\widehat x)}\colon [0,1]\to [0,1]$ are $C^1$ injective interval maps specified in Section~\ref{sec:1d}.
We also assume that the rate of expansion (contraction) of the horseshoe $\Gamma$ is stronger than any expansion (contraction) of $f_{\varpi(\widehat x)}$.
In this way the $DF$-invariant splitting $E^\sst\oplus E^c\oplus E^\uut$ given by
\begin{equation}\label{e.splitt}
    E^\sst\eqdef\bR\times\{(0,0)\}, \,
    E^c\eqdef\{(0,0)\}\times\bR, \,
    E^\uut\eqdef\{0\}\times\bR\times\{0\}
\end{equation}
is dominated and $E^{ss}$ and $E^{uu}$ are uniformly hyperbolic. We consider the maximal invariant set $\Lambda$ of $F$ in the cube $\CC=[0,1]^3$
\begin{equation}\label{def:Lambda}
	\Lambda\eqdef \bigcap_{j\in\ZZ}F^{-j}(\CC).
\end{equation}

Let us emphasize that \emph{transitivity} is a key property in our setting, as even in the hyperbolic case naturally there can occur phase transitions related to the existence of several basic pieces. Note also that there are phase transitions in the critical case associated to failure of transitivity related to renormalizable dynamics, see~\cite{Dob:09}.

We note that the $k$ phase transitions are obtained from $k$ corresponding gaps in the spectrum of the central Lyapunov exponents. The intervals between the spectral gaps correspond to ``lateral horseshoes'' which are associated to sub-shifts, exhibiting appropriate Birkhoff averages and entropy. This will imply that these transitions are rich.
Observe that the support of the relevant equilibrium states for $t\,\varphi$ jumps between invariant subsets as $t$ moves through the parameter of the phase transition.

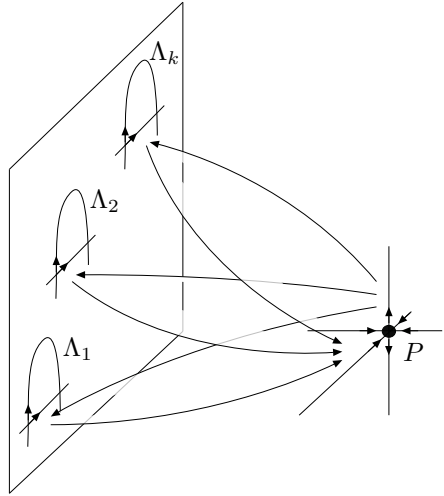
\begin{figure}
\begin{minipage}[c]{\linewidth}
\centering
\vspace{0.5cm}
\begin{overpic}[scale=.35,
  ]{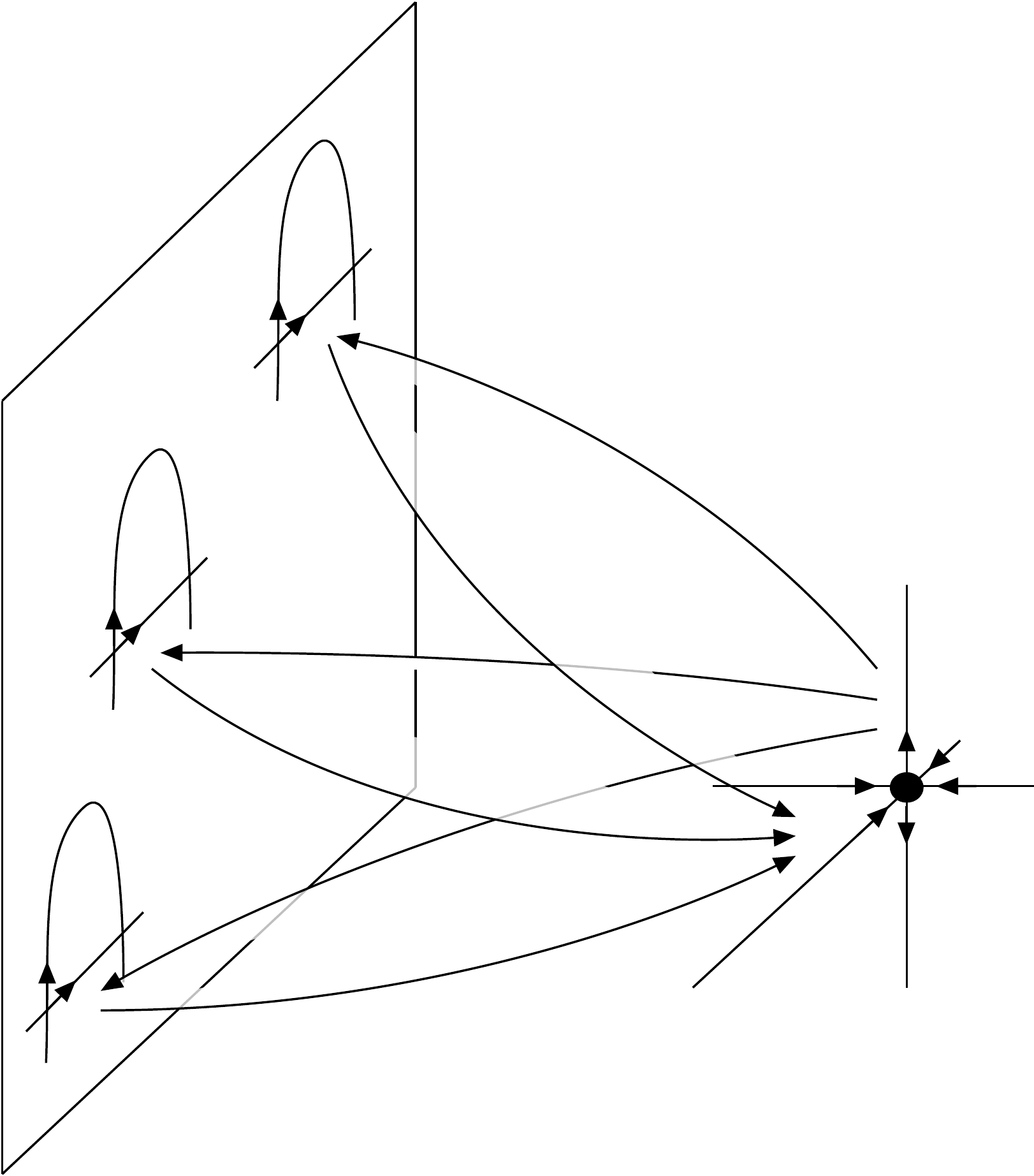}
      \put(80,27){\small$P$}	
      \put(11,28){\small$\Lambda_1$}	
      \put(16.5,58){\small$\Lambda_2$}	
      \put(28.5,88){\small$\Lambda_k$}	
  \end{overpic}
\caption{Heterodimensional cycles and connections}
\label{fig.sketch}
\end{minipage}
\end{figure}

The set $\Lambda$ that we consider is a special type of transitive set, called a homoclinic class. It was shown in~\cite{AbdBonCroDiaWen:07} that $C^1$-generically such classes have a convex Lyapunov spectrum. Thus, to obtain these gaps, we need to consider non-generic situations where the homoclinic class contains saddles of different types related by heterodimensional cycles (see~\cite[Chapter 6]{BonDiaVia:05}). Our study is also motivated by~\cite[Question 6.4]{BonDia:12} about the existence of multiple spectral gaps for homoclinic classes.

Two transitive hyperbolic sets of different \emph{index} (dimension of the unstable manifold) are related by a \emph{heterodimensional cycle} if their invariant manifolds meet cyclically (there are heteroclinic orbits going from one of the sets to the other and vice versa). Let us emphasize that the existence of such cycles is a typical feature in the partially hyperbolic setting. The understanding of the interplay of the sets in the cycle and the ``dominating'' hyperbolicity is a key step in the understanding of the global dynamics.

The heterodimensional cycle studied here are provided by a cycle condition in the fiber dynamics (see condition (F1)) which plays a somewhat similar role as the post-critical point in the quadratic family -- typical orbits only slowly approach to cycle sets which correspond to the critical point and the post-critical and give rise to some transient behavior and hence to spectral gaps.

A key feature of our construction is the existence of a cycle configuration involving horseshoes $\Lambda_1,\ldots,\Lambda_k\subset\Lambda$ (expanding in the central direction) and a saddle $P$ (contracting in the central direction) depicted in Figure~\ref{fig.sketch}.  Each horseshoe is, in some sense, an ``exposed'' parts of the dynamics. Each one  has different averages and thus is responsible for an interval of exponents in the central Lyapunov spectrum and these intervals are separated by gaps. These horseshoes play the role of exposed piece of the dynamics similar to the critical/post-critical point in the quadratic family.

Let us now contrast our non-critical example (the central dynamics does not have critical points) with the critical case of the quadratic family. Observe that the dynamics of a complex rational map on its (transitive) Julia set can have at most two phase transitions (\cite[Main Theorem]{PrzRiv:11}). Further,~\cite{CorRiv:a,CorRiv:b} present examples of complex quadratic polynomials having a first order phase transition with a unique associated equilibrium state  and with a phase transition with infinitely differentiable pressure function and no equilibrium state associated to the transition parameter, respectively.

In principle the number of phase transitions in a system can be arbitrarily large.
The example~\cite{Sar:00} of a topologically mixing countable Markov shift with a one-parameter family of piecewise constant potentials possesses a positive Lebesgue measure Cantor set of parameters at which the pressure is not analytic. However, in this example there do not exist equilibrium states associated to the parametrized potentials. The phase transition coincides with the change of the system from being recurrent to being transient (see also~\cite{IomTod:} for further discussion and references).

This paper is organized as follows. In Section~\ref{s:base} we select certain sub-shifts in the base dynamics that give rise to the subsets $\Lambda_1,\ldots,\Lambda_k$. In Section~\ref{sec:2.2} we define the fibre maps and complete the definition of the diffeomorphism $F$. In this section we present some topological properties of the underlying iterated function system and explain how these properties imply the transitivity of $\Lambda$. Finally, in Section~\ref{s:PROOF} we prove Theorem~\ref{t1}.

\section{Base dynamics}\label{s:base}

We now go to the details of the construction of $F$. The first step is the choice of appropriate sub-shifts in the base (horseshoe). The main result of this section is Proposition~\ref{p:main} stating differentiability properties of the pressure that is the symbolic counterpart of the one in Theorem~\ref{t1}.

\subsection{Construction of sub-shifts}\label{s:22}

We now choose a sub-shift $\bH$ of $\Sigma_{02}=\{0,2\}^\bZ\subset\Sigma_{012}$ that is a union of a finite number of disjoint topologically mixing pieces $H_\ell$, each of which having a different frequency of the symbol $2$.
We define the sub-shifts $H_\ell$ recursively.

We first give a general construction. Given a triple $(i,j,L)$ of integers $0\le i<j\le L$, we say that a finite word $(\xi_0\ldots \xi_{L-1})\in\{0,2\}^L$ is \emph{$(i,j,L)$-admissible} if
\[
	\#\big\{m\in\{0,\ldots,L-1\}\colon \xi_m=2\big\}\in\{i,\ldots,j\}.
\]
A finite word $(\xi_0 \ldots \xi_{n-1})$, $n>L$, is said to be \emph{$(i,j,L)$-admissible} if every sub-word $(\xi_k \ldots \xi_{k+L-1})$, $k=0,\ldots,L-n$, of length $L$  is admissible.

Fix some positive integer $k$. We define a finite sequence of triples $\cT_k=(i_\ell,j_\ell,L_\ell)_{\ell=1}^k$ of positive integers $i_\ell<j_\ell<L_\ell$ as follows.
We start with
\[
	0= i_1< j_1=L_1-(3k+1)<L_1.
\]	
Consider a strictly increasing sequence $(L_\ell)_{\ell=1}^{k}$.
For $\ell\ge1$ we choose recursively $(i_{\ell+1},j_{\ell+1},L_{\ell+1})$ with
\begin{equation}\label{eq:choicess}
\begin{split}
	j_{\ell+1} = L_{\ell+1}- (3k+1-3\ell),\quad
	i_{\ell+1} = j_{\ell+1}-1.
\end{split}	
\end{equation}
By construction, if the numbers $L_\ell$ are large enough, we have
\begin{equation}\label{eq:choice}
	L_\ell<i_{\ell+1}< j_{\ell+1}<L_{\ell+1},\quad
	L_{\ell+1}-i_{\ell+1} < \min_{n=1,\ldots,\ell}(L_n-j_n)-1.
\end{equation}
We will specify the sequence $(L_\ell)_{\ell=2}^{k}$ in two steps in Sections~\ref{ss.virgemsanta-1} and~\ref{ss.virgemsanta-2}.

For a fixed sequence of triples,
we define for each $\ell\ge1$ the sets
\[
	G_\ell
	\eqdef \Big\{\xi\in\Sigma_{02}\colon
		(\xi_0\ldots\xi_{L_\ell-1}) \text{ is }(i_\ell,j_\ell,L_\ell)
	\text{-admissible}\Big\}
\]
and
\[
	H_\ell
	\eqdef \Big\{\xi\in\Sigma_{02}\colon
	\sigma^n(\xi)\in G_\ell
	\text{ for every }n\in\bZ\Big\}
	= \bigcap_{s\in\bZ}\sigma^s(G_\ell).
\]
We consider also
\begin{equation}\label{def:Hk}
	\bH\eqdef
	\bigcap_{s\in\bZ}\sigma^s\Big(G_1\cup\ldots\cup G_k\Big)
	\supset H_1\cup\ldots\cup H_k.
\end{equation}

Given $t\in\bR$, consider the continuous potential $\phi\colon\Sigma_{012}\to\bR$ given by
\begin{equation}\label{eq2}
	\phi(\xi)\eqdef
	\begin{cases}
	 -\log\beta_0&\text{ if }\xi_0=0,\\
	 -\log\gamma&\text{ if }\xi_0=1,\\
	 -\log\beta_2&\text{ if }\xi_0=2.
	 \end{cases}
\end{equation}
The following is our main technical result about the topological pressure of $t\,\phi$ on the above defined sub-shifts (the proof is in Section~\ref{s:3.32}). We will study the topological pressure $\bP(t)= P_{\sigma|\bH}(t\,\phi)$ of $t\,\phi$ with respect to $\sigma\colon\bH\to\bH$.
By definition,
\[
	\bP(t)
	=\max_{\ell=1,\ldots,k}P_{\sigma|H_\ell}(t\,\phi)
\]

\begin{proposition}\label{p:main}
		There exist numbers $0>t_0>t_1>\cdots>t_k$ such that for every $\ell\in\{1,\ldots,k\}$ and every $t\in(t_\ell,t_{\ell-1})$ we have
	\[
		\bP(t)=P_{\sigma|H_\ell}(t\,\phi)
		>\max_{j\ne\ell}P_{\sigma|H_j}(t\,\phi).
	\]	
	Moreover, $\bP$ is real analytic in each interval $(t_\ell,t_{\ell-1})$.
	Further,  the left/right derivatives satisfy
	\[
		D^-\bP(t_\ell)<D^+\bP(t_\ell)
	\]
	and there exist equilibrium states $\mu_\ell^\pm$ for ${t_\ell}\,\varphi$ with respect to $\sigma|_{\bH}$ that both have positive entropy and Birkhoff average $-D^\pm \bP(t_\ell)$.
\end{proposition}

In the following subsections we specify certain averaging properties of the sub-shift guaranteed by an appropriate choice of the sequence $(r_\ell)_{\ell=1}^k$. However, we need to start from topological properties of $H_k$.

\subsection{Topological properties}

\begin{lemma} \label{mixtime}
Let $A,B$ be $(i,j,L)$-admissible. Then there exist a word $C$ of length not greater than $2L\cdot \min\{j, L-i\}$ such that $ACB$ is $(i,j,L)$-admissible.
\end{lemma}
\begin{proof}
Without weakening of assumptions we can assume that $A,B$ are of length $L$. Note that if $D$ and $E$ are $(i,j,L)$-admissible words of length $L$ that differ at most in one place then the word $DE$ is also $(i,j,L)$-admissible. So, when we take a sequence of $(i,j,L)$-admissible words $D_0=A$, $D_1,\ldots,D_{n-1}$, $D_n=B$ such that $D_m$ and $D_{m+1}$ differ in only one place then the word $AD_1D_2\ldots D_{n-1}B$ is $(i,j,L)$-admissible. As $A$ and $B$ differ in at most $2\min(j,L-i)$ places, the assertion follows.
\end{proof}

\begin{corollary} \label{mixing}
$(H_\ell,\sigma)$ is topologically mixing for every $\ell=1,\ldots,k$.
\end{corollary}

\subsubsection{Birkhoff averages}\label{ss.virgemsanta-1}

Given constants $\gamma\in(0,1)$ and $1<\beta_2<\beta_0$ (these constants will be specified in Section~\ref{sec:2.2}) and a triple $(i,j,L)$, we associated the following numbers
\begin{equation}\label{here}\begin{split}
		\alpha^-(i,j,L)&\eqdef
			\frac{i\log\beta_0+(L-i)\log\beta_2}
				{L}
				= \log\beta_0+\frac{L-i}{L}\log\frac{\beta_2}{\beta_0},\\
		\alpha^+(i,j,L)&\eqdef	
			\frac{j\log\beta_0+(L-j)\log\beta_2}
				{L}
				=  \log\beta_0+\frac{L-j}{L}\log\frac{\beta_2}{\beta_0}.
\end{split}\end{equation}
We write $\alpha^\pm_\ell\eqdef\alpha^\pm(i_\ell,j_\ell,L_\ell)$.
Observe that to obtain $\alpha^+_\ell<\alpha^-_{\ell+1}$ it is enough to have
$(L_\ell-j_\ell)/L_\ell>(L_{\ell+1}-i_{\ell+1})/L_{\ell+1},$
which follows from \eqref{eq:choicess}.

\subsubsection{Entropy constants}\label{ss.virgemsanta-2}

Associated to the sequence of triples there are also defined entropy constants $h_\ell$ that we will estimate in the following.

\begin{proposition} \label{prop:entropy}
There exist constants $K_1, K_2>0$ such that
\[
K_1 \frac {\log L_\ell} {L_\ell} \leq h_\ell \leq K_2 \frac {\log L_\ell} {L_\ell}.
\]
\end{proposition}
\begin{proof}
We can assume that $\ell >1$.
Note that by a slight modification of~\cite[Theorem 7.13 (i)]{Wal:81} the entropy $h_\ell\eqdef h(\sigma|_{H_\ell})$  of the sub-shift satisfies
\[
	h_\ell=\lim_{n\to\infty}\frac 1 n \log \theta(n),
\]
where $\theta(n)$ denotes the number of all cylinders $[\xi_0\ldots\xi_{m-1}]$, $m\le n$, in $H_\ell$. Note that in particular the limit exists and it is enough to consider some subsequence.

Observe that the number of cylinders of length $L_\ell$ in $H_\ell$ is given by
\[
	 C_\ell\eqdef \frac{L_\ell!}{i_\ell!(L_\ell-i_\ell)!}+
		\frac{L_\ell!}{j_\ell!(L_\ell-j_\ell)!}.
\]	
Note that

\[
\frac {j_\ell^{L_\ell - j_\ell}} {(L_\ell-i_\ell)!} \leq C_\ell \leq L_\ell^{L_\ell-i_\ell},
\]
hence (up to a multiplicative constant) we can write

\begin{equation} \label{eq:cell}
\log C_\ell \approx \log L_\ell
\end{equation}
(we remind that $L_\ell-i_\ell$ is bounded by $3k-1$). 

By Lemma~\ref{mixtime} for every pair of $(i_\ell,j_\ell,L_\ell)$-admissible finite words $A,B$ there is a finite word $C$ of length $m\le 2L_\ell (L_\ell-i_\ell)$ such that $ACB$ is $(i_\ell,j_\ell,L_\ell)$-admissible.

Hence, for every $n\ge1$ the number $\theta(L_\ell\cdot(1+n(L_\ell-i_\ell+1)))$ of cylinders of length at most  $L_\ell\cdot(1+n(L_\ell-i_\ell+1))$ is bounded from below by
\[
	{C_\ell}^n\ge\left(\frac{L_\ell!}{i_\ell!(L_\ell-i_\ell)!}\right)^n.
\]
In particular,
\[
	h_\ell \ge
	\frac{n}{L_\ell\cdot(1+n(L_\ell-i_\ell+1))}\log C_\ell.
\]
Substituting \eqref{eq:cell} we get the lower bound.

Similarly, we can estimate the number of periodic orbits of period $n\,L_\ell$ from above by ${C_\ell}^n$, which gives the upper bound.
\end{proof}

The above enables us to further specify the choice of $(L_\ell)_{\ell=2}^k$ in such a way that for each $\ell=1,\ldots,k-1$ we have
\begin{equation}\label{eq:chataalunoa}
	h_{\ell+1} < h_\ell.
\end{equation}

\subsubsection{Cones}

Associated to the triples there are cones $\cC_\ell$ defined by
\[
	\cC_\ell\eqdef \big\{(t,p)\in\bR^2\colon
	 h_\ell-\alpha_\ell^-\cdot t\le p \le h_\ell-\alpha_\ell^+\cdot t \big\}.
\]
Observe their geometric meaning: the graph of the pressure function $t\mapsto P_{\sigma|H_\ell}(t\,\phi)$ lies inside $\cC_\ell$, see Section \ref{s:3.32}.

For $\ell=1,\ldots,k-1$ let
\begin{equation}\label{zezez}
	\tau_\ell^- \eqdef - \frac{h_\ell-h_{\ell+1}}
					{\lvert  \alpha_{\ell+1}^- -  \alpha_\ell^+\rvert},
	\quad
	\tau_\ell^+ \eqdef - \frac{h_\ell - h_{\ell+1}}
					{\lvert  \alpha_{\ell+1}^+-\alpha_\ell^-\rvert} .
\end{equation}

\begin{remark}\label{r:erwartung}
By the above choices, we have $h_\ell>h_{\ell+1}$ and $\alpha_\ell^+<\alpha_{\ell+1}^-$. By construction, two consecutive cones $\cC_\ell$ and $\cC_{\ell+1}$ intersect each other in a rectangle that projects to the first coordinate to the interval $[\tau_\ell^-,\tau_\ell^+]$ (compare Figure~\ref{fig.11}).
\end{remark}

\begin{figure}
\begin{minipage}[c]{\linewidth}
\centering
\vspace{0.5cm}
\begin{overpic}[scale=.45,
  ]{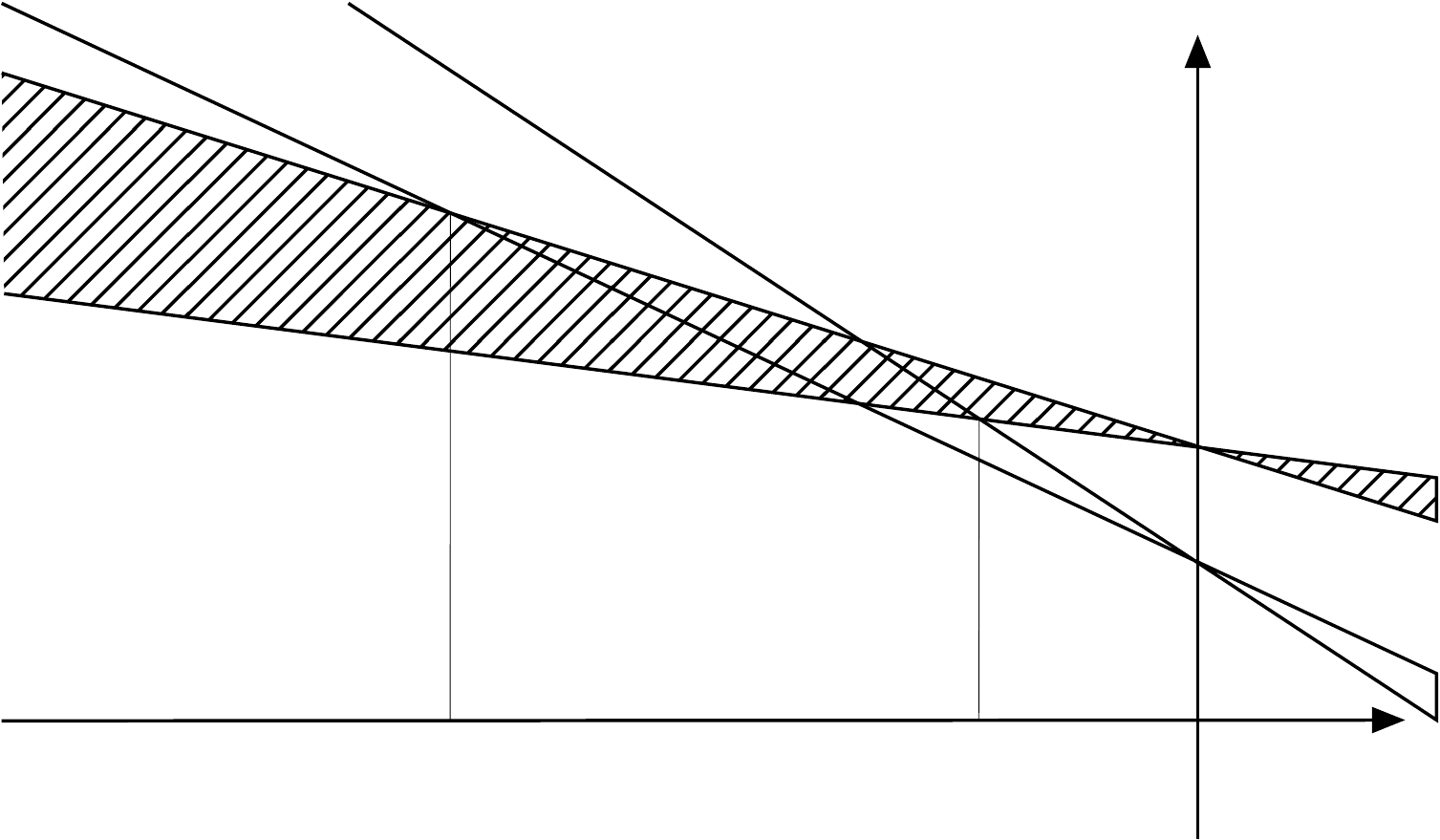}
    \put(98,2){\small$t$}	
    \put(85,30){\small$h_\ell$}	
    \put(102,22){\small$\cC_\ell$}	
    \put(102,8){\small$\cC_{\ell+1}$}	   	
    \put(35,2){\small$\tau_\ell^-$}	
    \put(68,2){\small$\tau_\ell^+$}	
  \end{overpic}
  \vspace{0.5cm}
\end{minipage}
\caption{Intersection of cones}
\label{fig.11}
\end{figure}
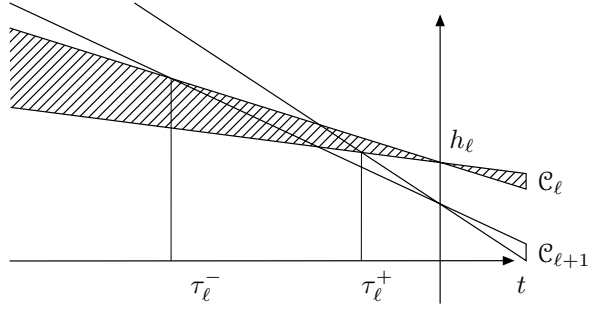

Given $\widetilde\beta\in(1,\beta_2)$ specified  in Section~\ref{sec:2.2}, let us define
\begin{equation}\label{e:important}
	t_0\eqdef \frac{\log 3 - \log 2}{\log\widetilde\beta - \log\beta_0}.
\end{equation}
This number corresponds to the point of intersection of the straight lines $(t,\log 2-t\log\beta_0)$ and $(t,\log3-t\log\widetilde\beta)$. This number will be  important in Section~\ref{s:PROOF}.

\subsubsection{Choice of the triples}

We now specify the choice of triples $\cT_k$.

\begin{lemma}\label{l:2}
	Given any $k\in \mathbb{N}$, there are triples
	$\cT_k=(i_\ell,j_\ell,L_\ell)_{\ell=1}^k$ satisfying conditions~\eqref{eq:choicess} in such a way that for every $\ell\in \{1,\ldots, k-1\}$
\begin{enumerate}
	\item $ h_{\ell+1}<h_\ell$,
	\item $\alpha_\ell^+<\alpha_{\ell+1}^-$, and
	\item $\tau_{\ell+1}^+<\tau_\ell^-<\tau_1^+<t_0$.
\end{enumerate}
\end{lemma}

\begin{proof}
Recall that by the choice~\eqref{eq:choicess} for every $\ell$ we have $\alpha_\ell^+<\alpha_{\ell+1}^-$.
Note that if $L_{\ell+1}$ is sufficiently big then by Proposition~\ref{prop:entropy}
we have
\[
	h_{\ell+1} \ll h_\ell
\]
and by definition~\eqref{here} we have
\[
\beta_0-\alpha_{\ell+1}^+ \ll \beta_0 - \alpha_\ell^-.
\]
Hence, for $L_{\ell+1}$ big enough we have
\[
	\lvert\tau_\ell^+\rvert
	\approx \frac {h_\ell} {\beta_0-\alpha_\ell^-}
	\approx \frac {\log L_\ell/L_\ell} {1/L_\ell} = \log L_\ell.
\]
Hence, $\lvert\tau_\ell^+\rvert$ can be made arbitrarily large.

With the above we can proceed recursively. Observe that for $\ell=1$ we have $\tau_1^+<t_0$ if $L_2$ is large enough. For $\ell\ge1$ suppose that we have already constructed $(i_m,j_m,L_m)$ for every $m=1,\ldots,\ell$ satisfying Properties 1--3. By definition, $\tau_\ell^-<\tau_\ell^+$. If $L_{\ell+1}$ is large enough then $\tau_{\ell+1}^+<\tau_\ell^-$.
\end{proof}

\subsection{Properties of the sub-shifts}\label{heree}

By construction, the sets $H_\ell$ all are pairwise disjoint and $\sigma$-invariant. Observe that $i_1=0$ implies that $0^\bZ\subset H_1$. It turns out that $\bH=\bigcup H_\ell$. 

\begin{lemma}\label{l:hetero}
	If $\xi\in \bH\cap G_\ell$ then $\sigma^s(\xi)\notin G_n$ for all $n\neq\ell$  and all $s\ge0$.
\end{lemma}

\begin{proof}
	If $\xi\in \bH\cap G_\ell$ then in the word $\xi_0\ldots \xi_{L_\ell-1}$ the symbol 0 appears either at least $3(k-\ell) +4$ times (if $\ell=1$) or precisely $3(k-\ell)+4$ or $3(k-\ell) +5$ times (if $\ell >1$). We know that $\sigma \xi\in G_m$ for some $m$.

If $m<\ell$ then it means that in the word $\xi_1\ldots \xi_{L_m}$ the symbol 0 appears at least $3(k-m)+4>3(k-\ell) +5$ times, which is impossible because this word is a sub-word of $\xi_0\ldots \xi_{L_\ell-1}$. If $m>\ell$ then it means that in the word $\xi_1\ldots \xi_{L_m}$ the symbol 0 appears no more than $3(k-m)+5\leq 3(k-\ell)+2$, which is impossible because this word contains all the symbols except the first one from $\xi_0\ldots \xi_{L_\ell-1}$ and the latter word has at least $3(k-\ell)+4$ 0's. Hence, $m=\ell$ and we proceed by induction. 
\end{proof}

Hence,
\[
H_\ell=\bH\cap G_\ell.
\]

\subsection{Proof of Proposition~\ref{p:main}}\label{s:3.32}

Given $\xi\in\Sigma_{012}$, consider the \emph{lower} and \emph{upper Birkhoff averages} of $t\,\phi$ at $\xi$ with respect to $\sigma$ given by
\[
		\underline\chi(\xi)\eqdef
		\liminf_{n\to\infty}\frac t n \sum_{k=0}^{n-1}\phi(\sigma^k(\xi)),\quad
		\overline\chi(\xi)\eqdef
		\limsup_{n\to\infty}\frac t n \sum_{k=0}^{n-1}\phi(\sigma^k(\xi)).
\]	
The definition of $H_\ell$ implies that for every $\ell\ge1$ and every $\xi\in H_\ell$ we have
\begin{equation}\label{formula}
	\underline\chi(\xi),\overline\chi(\xi)\in \big[\alpha^-_\ell,\alpha^+_\ell\big].
\end{equation}

Consider the set $\bH$ defined in~\eqref{def:Hk}. We will study the topological pressures of the potential $t\,\phi$ with respect to the shift maps $\sigma\colon H_\ell\to H_\ell$ and $\sigma\colon \bH\to \bH$. Observe that  for every $t\in\bR$ we have
\[
	\bP(t)= P_{\sigma|\bH}(t\,\phi)
	=\max_{\ell=1,\ldots,k}P_{\sigma|H_\ell}(t\,\phi).	
\]
To locate the graph of the pressure function $\bP(t)$ note that
\begin{equation}\label{ref:entropie}
	P_{\sigma|H_\ell}(0)=h_\ell
\end{equation}	
and that any sub-gradient of the pressure function is in $[\alpha^-_\ell,\alpha^+_\ell]$. Thus, using the cones defined above, we get that
\begin{equation}\label{ea:before}
	(t, P_{\sigma|H_\ell}(t\,\phi))\in\cC_\ell \quad\text{ for every }t\in\bR.
\end{equation}

Note that for every $t\in\bR$ the  potential $t\,\phi$ is H\"older continuous. By Corollary~\ref{mixing}, for each $\ell$ the sub-shift $\sigma\colon H_\ell\to H_\ell$ is uniformly expanding and topologically mixing, thus the function $t\mapsto P_{\sigma|H_\ell}(t\,\phi)$ is real analytic~\cite{Rue:78}. In particular, by the above, derivatives  are in the interval $[\alpha_\ell^-,\alpha_\ell^+]$.

Note that $\bP$ is convex and thus differentiable for all but at most countably many $t$, and the left and right derivative $D^-\bP(t)$ and $D^+\bP(t)$ are defined for all $t$. Moreover, $\bP$ is differentiable at $t$ if, and only if, all equilibrium measures for $t\,\phi$ have the same average $-\bP'(t)$ (refer to~\cite[Section 9.5]{Wal:81} for more details).

Write $P_\ell(t)= P_{\sigma|H_\ell}(t\,\phi)$.
By the remarks above, we have $P'_\ell(t)\in[\alpha_\ell^-,\alpha_\ell^+]$. Since the intervals $[\alpha_\ell^-,\alpha_\ell^+]$ are pairwise disjoint, the graphs of $P_\ell$ and $P_{\ell+1}$ intersect in exactly one point $t_\ell$.

The choice of the cones implies that $\bP(t)=P_\ell(t)$ for $t\in [t_\ell,t_{\ell-1}]$ and the choice of $\tau_\ell^\pm$ implies that $t_\ell\in[\tau_\ell^-,\tau_\ell^+]$. Hence, the pressure changes from the cone  $\cC_\ell$ to $\cC_{\ell+1}$. Since in each cone the slope of the pressure is in the interval given by the ``opening"  of the cone and these ``openings" are disjoint this implies that at $t_\ell$ the pressure is not differentiable.

Fix $\ell$ and let us focus on the sub-dynamics of $\sigma|_{H_\ell}$. First, let us obtain an equilibrium state with the claimed properties. Let $\mu_\ell^+$ be an weak$\ast$ accumulation point of $\mu_t$ as $t\searrow t_\ell$, where $\mu_t$ is the (unique) equilibrium state for $t\,\phi$ with respect to $\sigma|_{H_\ell}$. As above, we observe $t\int\phi\,d\mu_t=-P_\ell'(t)$.  In particular, we hence have
\[
	t\int\phi\,d\mu_\ell^+=\lim_{t\searrow t_\ell}t\int\phi\,d\mu_t
	=\lim_{t\searrow t_\ell}-P_\ell'(t)=-D^+\bP(t_\ell).
\]	
Upper semi-continuity of the entropy function implies that $\mu_\ell^+$ is an equilibrium state for $t\,\phi$ with respect to $\sigma|_{H_\ell}$. By a standard argument we can assume that $\mu_\ell^+$ is ergodic. Thus, typical Birkhoff averages are equal to $-D^+\bP(t_\ell)$, proving the first property.
 Observe that for every $t\in\bR$ we  have
\[
	h(\mu_t)=P_\ell(t)+tP_\ell'(t)
\]
or, in other words, $h(\mu_t)$ is the intersection of the tangent line to the pressure $P_\ell$ at $t$ with the vertical axis.
Assume, by contradiction, that $h(\mu_\ell^+)=0$.
As $P_\ell$ is strictly convex, this would imply that for $t<t_\ell$ the entropy $h(\mu_t)$ would be negative, which is a contradiction.

Analogous arguments apply to $\mu_\ell^-$.

Finally, Lemma~\ref{l:2} item 3 implies that $t_1\le\tau_1^-<t_0$, ending the proof of the proposition. \hfill$\square$

\section{Smooth realization}\label{sec:2.2}

We now define the diffeomorphism $F$.
We consider maps $\{f_0,\widetilde f_0,f_1,f_2,\widetilde f_2\}$ as in the Figure~\ref{fig.1}, see precise definition in Section~\ref{sec:1d}.
It will be sufficient to define those maps in some  neighborhood of $[0,1]$.
\begin{figure}
\begin{minipage}[h]{\linewidth}
\centering
\vspace{0.5cm}
\begin{overpic}[scale=.40,
  ]{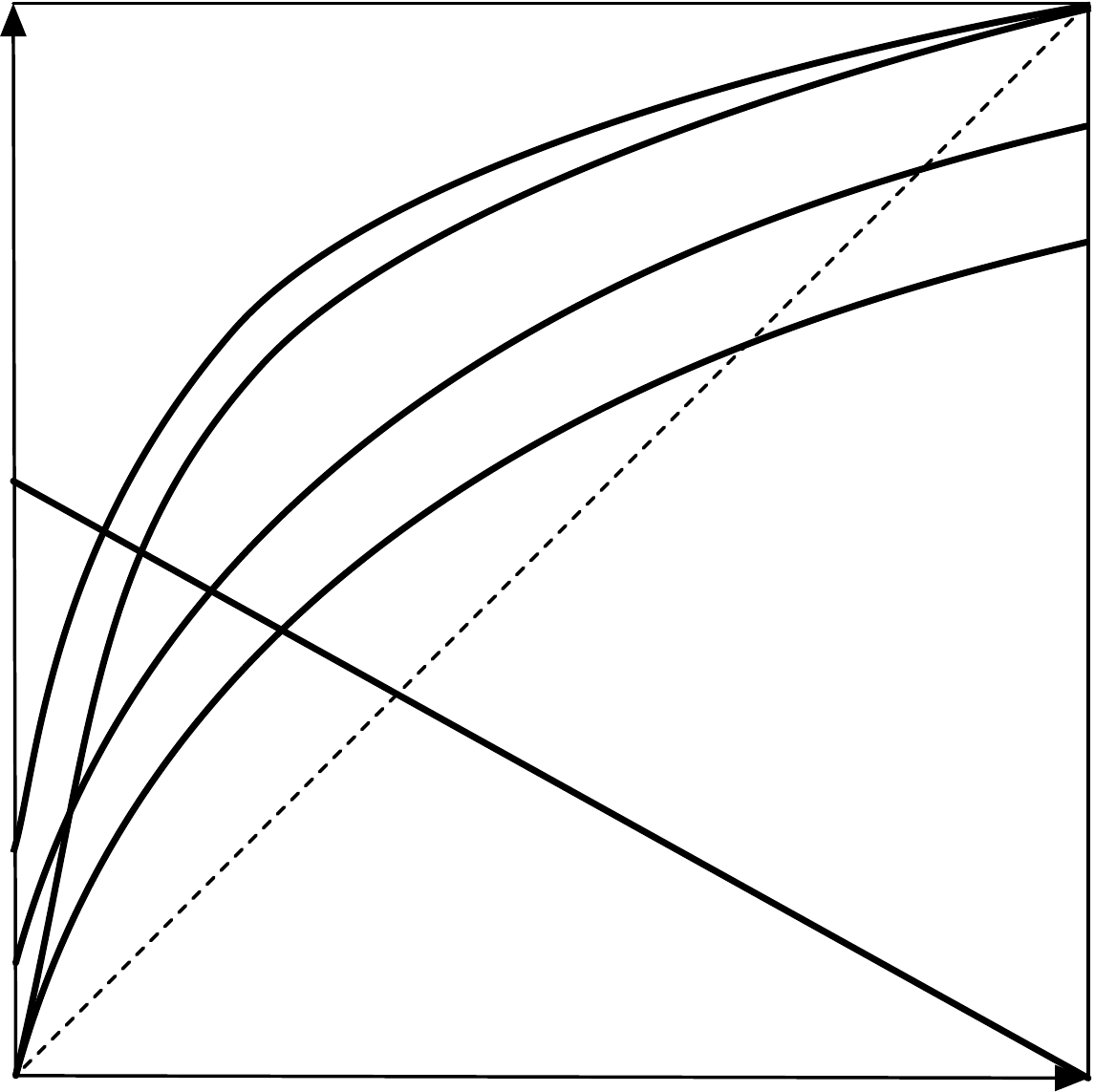}
     	\put(102,75){\small$f_2$}
     	\put(102,97){\small$f_0$}
    	\put(-10,25){\small$\widetilde f_2$}
     	\put(-10,10){\small$\widetilde f_0$}
	\put(70,25){$f_1$}
  \end{overpic}
\end{minipage}
\caption{The fiber maps $f_0,\widetilde f_0,f_1,f_2$, and $\widetilde f_2$.}
\label{fig.1}
\end{figure}
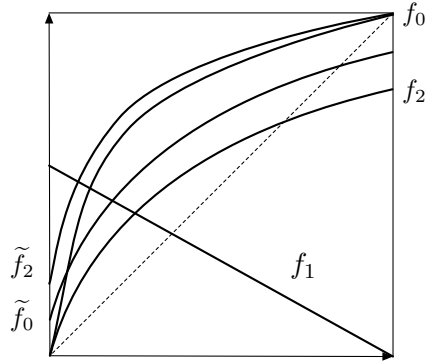

Fix $k\ge1$, the triples $(i_\ell, j_\ell,L_\ell)_{\ell=1}^k$, and the sets $(G_\ell)_{\ell=1}^k$ as in Section~\ref{s:22}.
Each cylinder $[\xi_0\ldots\xi_\ell]\in\{0,2\}^{\ell+1}$ associates a sub-cube $\widehat\CC_{[\xi_0\ldots\,\xi_\ell]}$ of $\widehat\CC=[0,1]^2$ defined as
the connected component of $\Phi^{-\ell}(\widehat\CC)\cap \widehat\CC$ containing $\varpi^{-1}([\xi_0\ldots\xi_\ell])$, where $\varpi$ is the conjugation map defined in Section~\ref{sec:1}. To produce a simple example, we will assume that $\Phi$ is affine in $\widehat\CC_{[i]}$.
To write the map $F$ in a compact way, write $\widetilde f_1=f_1$ and define $F$ as in~\eqref{e.defF} by
\[
	F(\widehat x,x)\eqdef (\Phi(\widehat x),f_{\varpi(\widehat x)}(x)),
\]
where the fibre maps are given by
\begin{equation}\label{def:almos}
   f_\xi( x)\eqdef
    \begin{cases}
        f_{\xi_0}(x)& \text{ if }
        \varpi( x)\in
        \bigcup_{\ell=1}^k\,  \, \bigcup_{\xi\in G_\ell} [\xi_0\ldots\xi_{L_\ell}],
        \\[0.1cm]
	\widetilde f_{\xi_0}(x)& \text{ otherwise}.
    \end{cases}
\end{equation}
Consider the sub-cubes $\CC_{[\xi_0\ldots\,\xi_m]}\eqdef \widehat \CC_{[\xi_0\ldots\,\xi_m]}\times [0,1]$. To complete the definition of $F$ in $\CC$ we take some appropriate $C^1$-continuation $F$ such that
$$
	F\Big(\text{int}\Big(\CC\setminus
	\bigcup_{\ell=1}^k\,\,\,\bigcup_{\xi=\varpi(\widehat x)\in G_\ell}
	\CC_{[\xi_0\ldots\,\xi_{L_\ell}]}\Big) \Big)\cap \CC=\emptyset.
$$
Note that the skew product $F$ is constant on each such sub-cube  $\CC_{[\xi_0\ldots\,\xi_{L_\ell}]}$. As each fibre map of the skew product is determined by at most $L_k$ symbols of the symbolic representation of $\widehat x$, the map $F$ is  an $L_k$-step skew product.
\begin{figure}
\begin{minipage}[c]{\linewidth}
\centering
\begin{overpic}[scale=.45]{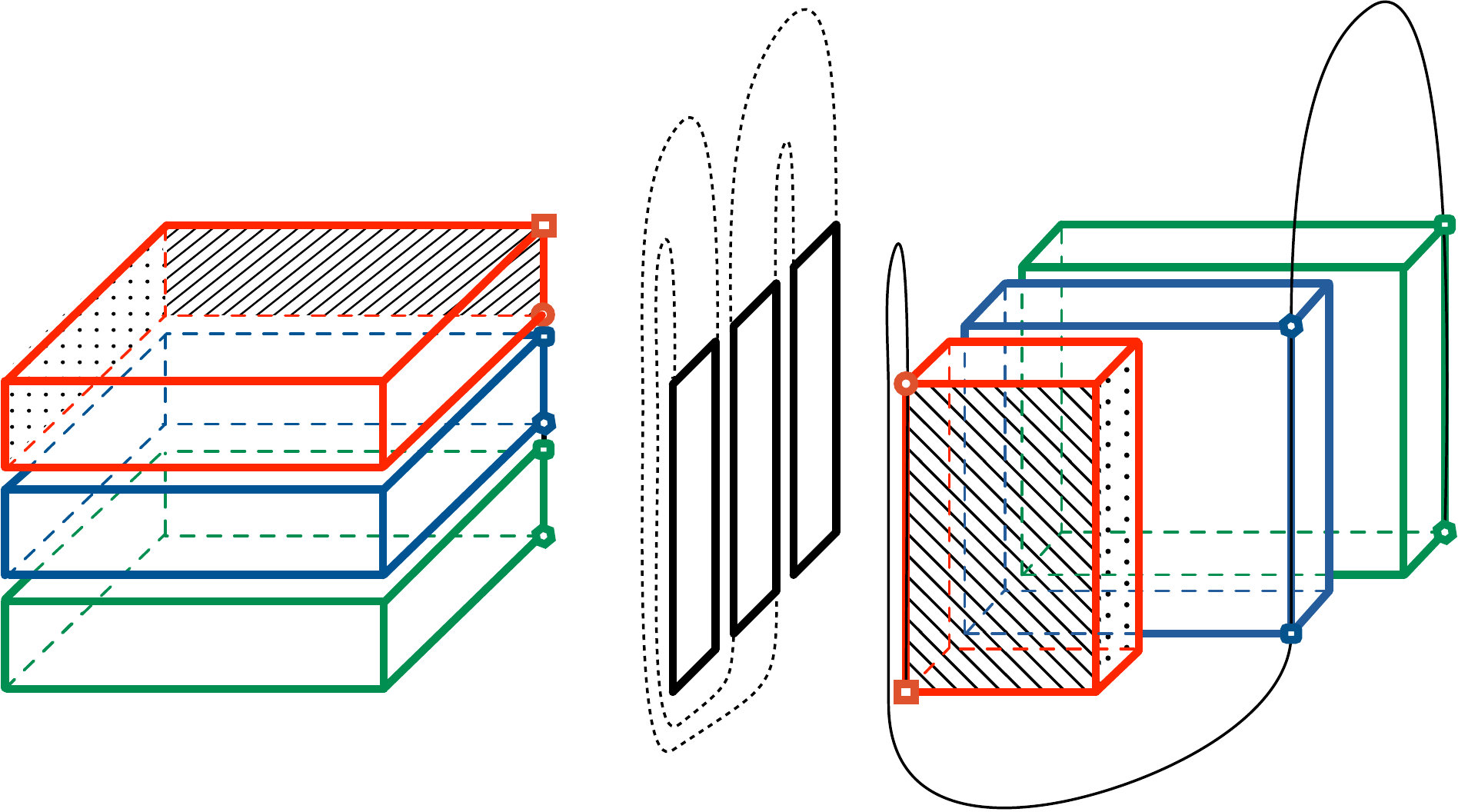}
  	\put(-7,17){$\CC_2$}	
  	\put(-7,9){$\CC_0$}	
  	\put(-7,25){$\CC_1$}	
 \end{overpic}
\caption{Construction of level 1 rectangles $\CC_{\xi_0}$ associated to cylinders $[\xi_0]$, $\xi_0=0,1,2$}
\label{fi.neu}
\end{minipage}
\end{figure}
Figure~\ref{fi.neu} illustrates roughly the construction of \emph{level-1} rectangles.

Recall the definition of $\bH \subset \Sigma_{02}$  in~\eqref{def:Hk}
and consider the lateral set
\begin{equation}\label{defLLLL}
	\Lambda_\bH
	\eqdef \Lambda\cap \bR^2\times\{0\}
	= \big\{ (\varpi^{-1}(\xi),0)\colon \xi\in\bH\big\}
\end{equation}
that contains $k$ disjoint proper horseshoes
$	\Lambda_\ell=\varpi^{-1}(H_\ell)$,
$\ell=1,\ldots,k$ mentioned in the introduction, see Figure~\ref{fig.sketch}. See also Section~\ref{ss.transitivelambda}.

\subsection{The underlying iterated function system}\label{sec:1d}

We start with the following conditions.

\begin{itemize}
\item[{(F0)}]
    The map $f_0$ is increasing and has exactly two hyperbolic fixed
    points in $[0,1]$, the point $q_0=0$ (repelling) and the point $p_0=1$ (attracting).
    Let $\beta_0=f_0^\prime (0)>1$ and $\lambda_0=f_0^\prime (1)\in
    (0,1)$.     \\[-0.3cm]
\item[{(F1)}]
    The map $f_1$ is an affine contraction $f_1(x)\eqdef \gamma\,(1-x)$ where $\gamma\in (\lambda_0,1)$. We denote by $p_1$ the attracting fixed point of $f_1$.
    Note that $f_1(1)=0$ (cycle condition).\\[-0.3cm]
\item[{(F2)}]
    The map $f_2$ is increasing and has two hyperbolic fixed
    points in $[0,1]$, the point $q_2=0$ (repelling) and the point $p_2\in(0,1)$ (attracting). Let $\beta_2=f_2^\prime (0)>1$.
\end{itemize}

Before going into further details, let us explain some heuristic ingredients.

A key point is the existence of a fundamental domain $J=[f_0^{-1}(b),b]\in(0,1)$, $b$ close to $0$, such that $f_0^N(J)$, for some large $N$, is close to $1$ and $f_1\circ f_0^N(J)\subset (0,b)$ and the restriction of $f_1\circ f_0^N$ to $J$ is uniformly expanding. This is the essential point for defining expanding itineraries (see Section~\ref{ss.1d}) and hence proving transitivity of $\Lambda$. This is guaranteed by the following.

\begin{itemize}
\item[{(F01)}]
	The derivative $f_0'$ is decreasing in $[0,1]$, in particular, $\lambda_0\le f_0^\prime(x)$ for all $x \in [0,1]$, and satisfies
	\[
	\gamma\, \lambda_0^3\, (1-\lambda_0)(1-\beta_0^{-1} )^{-1}>1.
	\]
\end{itemize}
Note that given $\gamma, \lambda_0\in(0,1)$, condition (F01) is clearly satisfied if $\beta_0>1$ is sufficiently close to $1$.

The maps $\widetilde f_0,\widetilde f_2$ are chosen as follows.
\begin{itemize}
\item[{(\~F02)}] The maps  $\widetilde f_0,\widetilde f_2$ are close to $f_0,f_2$ and satisfy $\widetilde f_0(0)>0$, $\widetilde f_0(1)=1$, and $\widetilde f_2(0)>0$. Moreover, for simplicity, we take $\widetilde f_0=f_0$ in $[b,1]$.
\end{itemize}

Finally, a key point is the existence of a gap in the spectrum (see Section~\ref{sec:spgap}). For simplicity, let us assume that $\log\beta_2$ is close to but smaller than $\log\beta_0$. Due to monotonicity of the maps $f_0,f_2$ and its derivatives, the only way that a point $x\in\Lambda\setminus\Lambda_\bH$ may have central Lyapunov exponent close to $\log\beta_2$ is when its orbit stays close to $\Lambda_\bH$ for a long time. But due to the cycle configuration, such an orbit previously stayed close to $1$ for a long time. The latter compensates the expansion and as a result the central Lyapunov exponent is smaller than some number $\log\widetilde\beta<\log\beta_2$ (see definition~\eqref{def:EE}).

The choice of $f_2,\widetilde f_2,\widetilde f_0$ is such that they are close to $f_0$ in such a way that the expansion they introduce do not destroy the spectral gap $(\log\widetilde\beta,\log\beta_2)$. The precise technical condition is literally (\emph{mutatis mutandi}) the same as (F012) in~\cite{DiaGelRam:11}, however stating the condition not only for $f_0,f_2$ but for $f_0,\widetilde f_0$ and $f_2,\widetilde f_2$, respectively%
\footnote{
\textbf{(F012)}
	We have $f_0^\prime(x),\widetilde f_0^\prime(x), f_2'(x),\widetilde f_2^\prime(x)\le\beta_0$ for all $x \in [0,1]$.
	The interval $H=[0,b]$, $b$ as above, is such that with $H'=f_1^{-1}(H)$ the  intersections
	\[
		f_1(H)\cap H ,
		\,\, f_1(H')\cap H',
		\,\, f_1([0,1])\cap H', 
		\,\, f_2([0,1])\cap H',
		\,\, \widetilde f_2([0,1])\cap H'
	\]
    are all empty. Moreover, we assume that
	\begin{gather*}	
	\beta' \eqdef \max\{f_0'(x),\widetilde f_0'(x), f_2'(x), \widetilde f_2'(x)
		\colon x\notin H\}
		<\beta_2,\\
		\beta_H\eqdef
		\min\{f_0'(x),\widetilde f_0'(x), f_2'(x),\widetilde f_2'(x)
			\colon x\in H\}<\beta_2,\\
		\lambda' \eqdef \max\{f_0'(x),\widetilde f_0'(x) \colon x\in H'\}<1,\\
    	\lvert\log\lambda_0\rvert\frac{\log\beta_0}{\log\beta_H}
	-\lvert\log\lambda'\rvert
	+\frac 2 3\log\beta_0<\frac 3 4  \log\beta'.
\end{gather*}
    Finally, let us also assume that with $L\ge1$ satisfying
    \[
    	L> 4\,(\lvert\log\lambda_0\rvert\,\log\beta_0)(\log\beta_H\,\log\beta')^{-1}
   \quad \text{ we have \quad}	
    \log((\beta_0)^L\,\gamma)(L+1)^{-1} < \log\beta'.
    \]
}.

\subsubsection{Expanding properties of the iterated function system}\label{ss.1d}

First, we provide more details on the IFS.
The maps $f_0, f_1$ are chosen (see \cite[Remark 7.2]{DiaGelRam:11}) such that there are numbers $\kappa>1$, $b\in (0,1)$ close to $0$ as above, and $N\in \NN$, and  such that every interval $J\subset [f_0^{-2}(b),b]$  has associated a natural number $n(J)\le N$ satisfying the following:
\[
	(f_1\circ f_0^{n(J)})(x) \subset (0,b]
	\quad\text{and}\quad
	\lvert(f_1\circ f_0^{n(J)})'(x)\rvert \ge \kappa, \quad
	\text{ for every }x\in J.
\]	
Let $m(J)\ge 0$ be the smallest number such that
\[
	(f_0^{m(J)}\circ f_1\circ f_0^{n(J)})(J) \cap (f_0^{-1}(b),b]\ne \emptyset.
\]
Noting that $f_0^{n(J)}(J) \subset [f_0^{N-2}(b),f_0^N(b)]$ we have that
\begin{equation}\label{e.itinerary}
	(f_1\circ f_0^{n(J)})(J) \subset [f_1\circ f_0^N (b), b].
\end{equation}
As $b$ is close to $0$ we have $f_0'(x)>1$ (recall (F0)). Thus
\[
	\lvert (f_0^{m(J)} \circ f_1\circ f_0^{n(J)})'(x)\rvert \ge \kappa,
	\quad \text{ for every } x\in J.
\]
In this way for  $J\subset [f_0^{-2}(b),b]$ we get its \emph{expanding itinerary} $(0^{n(J)}\,1\,0^{m(J)})$ and define its \emph{expanded successor} by
\[
	(f_0^{m(J)}\circ f_1 \circ f_0^{n(J)}) (J).
\]
Arguing inductively, one defines intervals $J_j\subset[f_0^{-2}(b),b]$ such that $J_0=J$ and $J_{j+1}$ is the expanded successor of $J_j$.
As we have
\[
	\lvert J_{j+1}\rvert
	\ge \kappa \,\lvert J_j\rvert
	\ge \kappa^{j+1}\, \lvert J_0\rvert,
\]
this provides a smallest number $\ell=\ell(J)$ such that $J_{\ell+1}$ is not contained in $[f_0^{-2}(b),b]$ and thus contains $[f_0^{-2}(b),f_0^{-1}(b)]$.
For the IFS associated to $\{f_0,f_1\}$ we use the standard cylinder notation: given a finite sequence $(\xi_0\ldots \xi_k)$ let
\[
	f_{[(\xi_0\ldots \xi_k)]}\eqdef f_{\xi_k}\circ\cdots\circ f_{\xi_0}.
\]
With this notation, the above construction is summarized as follows.

\begin{lemma}\label{l.xJ}
Associated to every interval $J\subset [f_0^{-2}(b), b]$ there is a finite sequence $\xi(J)=(\xi_0\ldots \xi_{k(\ell)})= (0^{n_0}\,1\, 0^{m_0}\ldots 0^{n_\ell}\,1\, 0^{m_\ell})$
such that
\begin{enumerate}
	\item $f_{[\xi(J)]}$ is uniformly expanding in $J$,
	\item $f_{[\xi(J)]}(J)$ contains $[f_0^{-2}(b),f_0^{-1}(b)]$.
\end{enumerate}
\end{lemma}

Applying the above, we obtain the following (see also \cite[Lemma 3.8]{DiaGel:12}).

\begin{corollary}\label{c.qstar}
	Given the interval $D=[f_0^{-2}(b), f_0^{-1}(b)]$ and its
	expanding sequence $\xi(D)$, there is a unique expanding fixed point $q_\ast\in D$ of $f_{[\xi(D)]}$.
	Moreover, the unstable manifold $W^u(q_\ast,f_{[\xi(D)]})$ contains  $D$.
\end{corollary}

For our purposes, a key property of the previous construction is that it \emph{only} involves expanding itineraries of intervals whose orbits are contained in $[f_1\circ f_0^{N} (b), f_0^N(b)]$ (this follows from \eqref{e.itinerary}) and this interval
is  disjoint from $\{0\}$. Thus, it continues to hold after replacing the map $f_0$ by any map $\widetilde{f}_0$ which coincides with $f_0$
in that interval (compare condition (\~F02)).

\subsubsection{Principal spectral gap}\label{sec:spgap}

Let us now provide details on the spectral gap. First, given a sequence $\xi\in\{0,1,2\}^\bZ$, for $n\ge0$ let
\[
	f_\xi^n\eqdef f_{\sigma^{n-1}(\xi)}\circ\cdots\circ f_{\xi},
\]
where $f_\xi$ is as in~\eqref{def:almos}. Given $x\in[0,1]$ and $\xi\in\Sigma_{012}$, the \emph{upper (forward) Lyapunov exponent} of $x$ with respect to $\xi$ is defined by
\[
	\overline\chi(\xi,x)\eqdef \limsup_{n\to\infty}\frac 1 n \log\,\lvert (f_\xi^n)'(x)\rvert.
\]
The lower exponent $\underline\chi$ is defined analogously taking $\liminf$ instead of $\limsup$.

Similarly as in~\cite{DiaGelRam:11}, we consider the set $\cE$ of `exceptional points' defined by
\begin{multline}\label{exceexce}
	\cE\eqdef \big\{(\xi,0)\colon\xi\in\bH\big\}\cup\\
	\big\{\big((0^{-\bN}.0^k\,1\,\xi^+),1\big),\big((0^{-\bN}\,1\,\xi_k.\xi^+),0\big)
		\colon k\ge 0,\xi_k\in\{0,2\}^k, \xi^+ 
		\big\}
\end{multline}
where in the second set $\xi^+$ is any one-sided sequence such that there is some $\xi^-$ with $\xi=(\xi^-.\xi^+)\in\bH$. That is, the set $\cE$ codes all points in the ``full lateral horseshoe'' $\{(\xi,0)\colon\xi\in\bH\}$ together with its stable manifold. An interpretation of $\cE$ in terms of heterodimensional cycles will be given below.

\begin{remark}[Gaps in the central spectrum]\label{remremrem}{\rm
	For pairs in the exceptional set $\cE$ the exponents naturally are in $[\log\beta_2,\log\beta_0]$, and are therefore beyond the gap.
However, by our hypotheses, as for~\cite[Proposition 4.2]{DiaGelRam:11} we obtain the following gap in the spectrum of Lyapunov exponents for non-exceptional points
\begin{equation}\label{def:EE}
	\exp\Big(\sup\big\{\overline\chi(\xi,x)\colon(\xi,x)\notin\cE\big\}\Big)
	\eqdef\widetilde\beta
	<\beta_2.
\end{equation}
Note that the choice of $H_\ell$ implies that
\[
	\big\{\underline\chi(\xi,0),\overline\chi(\xi,0)
		\colon \xi\in H_\ell\big\}
	=[\alpha_\ell^-,\alpha_\ell^+].
\]
Therefore, by the definition of $\Lambda_\ell$ and by Lemma~\ref{l:2}, besides the initial gap $(\log\widetilde\beta,\log\beta_2)$ there are $k-1$ further gaps $(\alpha_k^+,\alpha_{k-1}^-),\ldots,(\alpha_2^-,\alpha_1^+)$.
}\end{remark}

\subsection{Dynamics of the skew product diffeomorphism}\label{ss.transitivelambda}

In this section we derive some topological properties of the set $\Lambda$ and prove that this set is transitive. The transitivity will follow adapting methods in \cite{DiaGelRam:11}. Let us also observe that other topological properties of $\Lambda$ such as the abundance of non-trivial spines and the density of contracting (and also expanding) hyperbolic periodic points can be obtained following~\cite{DiaGel:12,DiaGelRam:}.

We start with a dynamical interpretation of the exceptional set~\eqref{exceexce}.

The orbit of any point $(\widehat x,0)\in \Lambda$ with $\varpi(\widehat x)\in \bH$ stays in $[0,1]^2\times\{0\}$.
Thus, clearly
\[
	\Lambda_\bH\subset\Lambda\cap ([0,1]^2\times\{0\})
\]	
Indeed $\Lambda_\bH$ is a proper invariant subset of  $\Lambda\cap ([0,1]^2\times\{0\})$, as we will see in the next paragraph.
Moreover, it is straightforward to see that
\[
	W^s(\Lambda_\ell,F)\cap W^u(P,F)\ne\emptyset
	\quad\text{ and }\quad
	W^u(\Lambda_\ell,F)\cap W^s(P,F)\ne\emptyset,
\]
$\ell=1,\ldots,k$, where $\Lambda_\ell=\varpi^{-1}(H_\ell)$, see Section~\ref{sec:2.2}.
This means that $P$ and $\Lambda_\ell$ are related by a heterodimensional cycle. We will see that the points in these intersections belong to $\Lambda$ and that they correspond precisely to the points coded by $\cE$.

The following proposition is the main result in this section.

\begin{proposition}\label{prop:locmax}
The set $\Lambda$ is transitive.
\end{proposition}

The main ingredient of the proof is the existence of expanding  itineraries for the iterated function system generated by the maps $f_0,f_1\colon [0,1]\to [0,1]$.
Let us note that the maps $f_2$, $\widetilde{f}_0$, $\widetilde{f}_2$ do not play any role  in these arguments.

Given the fixed point $q_\ast$  
and the expanding sequence $\xi=\xi(D)$ of the  interval $D=[f_0^{-2}(b), f_0^{-1}(b)]$ provided by Corollary~\ref{c.qstar}, set $\widehat q=\varpi^{-1}(\xi^\bZ)$.
Consider the associated periodic point $Q^\ast =(q^s,q^u,q_\ast)=(\widehat q,q_\ast )\in \Lambda$ of $F$.

\begin{remark}
	Note that the set $\Lambda$ is not locally maximal (recall that a closed invariant set $S$ is locally maximal if there is a small neighborhood $U$ of $S$ such that $S=\bigcap_{k\in\bZ}F^k(U)$). To see why this is so, just consider any point in $\{0\}\times\widehat\CC_1\times\{-\varepsilon\}$ for $\varepsilon>0$ small. This set contains a Cantor set whose forward orbit is contained in $\CC$ and whose backward orbit converges to $(0,0,0)$. Since this holds for every small $\varepsilon$, this prevents $\Lambda$ from being locally maximal.
	
	On the other hand, one can check that $\Lambda$ is the set of all non-wandering points of $F|_U$ provided that $U$ is sufficiently small such that the maps $\widehat f_0,\widehat f_2$ are both positive in this neighborhood.
	
	Situations where the non-wandering set is included in the locally maximal set are, in fact, typical in dynamics involving heterodimensional cycles. The are also present in other non-hyperbolic cases such as in the case of a parabolic point of a rational map  with petal dynamics where the only non-wandering point is the fixed point, which is not locally maximal because there are homoclinic trajectories inside any neighborhood of it (actually, for almost every point both $\alpha$-limit and $\omega$-limit is the fixed point).
\end{remark}

Given the neighborhood $\CC$ of $Q^\ast$, the closure of the points of transverse intersections of the stable and unstable manifolds of the orbit of $Q^\ast$ whose orbits are contained in $\CC$ is  called the \emph{homoclinic class relative} to $\CC$ and denoted by $H_\CC(Q^\ast,F)$. It is well known that relative homoclinic classes are transitive sets (see~\cite[Chapter 10.4]{BonDiaVia:05}). Thus Proposition~\ref{prop:locmax} follows from the next lemma.

\begin{lemma}\label{l:relhom}
	$\Lambda=H_\CC(Q^\ast,F)$.
\end{lemma}

\begin{proof}
This lemma is similar to~\cite[Proposition 3.3]{DiaGelRam:11} and is derived after modifying some of the arguments there, so we will just emphasize these modifications.
We start stating some topological properties of the invariant manifolds of $Q^\ast$. Note that
\begin{equation}\label{e.manifolds}
\begin{split}
	\{q^s\}\times [0,1]\times D
	&\subset W^u(Q^\ast , F) \quad \text{and}\\
	[0,1]\times \{(q^u, q_\ast )\}
	&\subset W^s(Q^\ast , F).
\end{split}
\end{equation}
An immediate consequence of the cycle configuration and Corollary~\ref{c.qstar} is that $W^u(Q^\ast,F)$ in its projection covers the whole square $\{0^s\} \times [0,1]\times (0,1)$. More precisely, for all $x\in (0,1)$ there are $x^s\in (0,1)$, $\varepsilon(x)>0$ with
\begin{equation}\label{e.unstablemanifold}
 \{x^s\} \times [0,1]\times [x-\varepsilon(x), x+\varepsilon(x) ]\subset
 W^u(Q^\ast,F),
\end{equation}

Take any point $X=(x^s, x^u,x)\in \Lambda$. In what follows let us assume that $x\in (0,1)$. The case $x\in \{0,1\}$ is indeed simpler and is treated identically as in \cite{DiaGelRam:11}, so we will omit this case. 

Given small $\delta>0$, consider the stable disk of $X$ 
$$
	\Delta^s_\delta(X)
	\eqdef [x^s-\delta, x^s+\delta]\times \{(x^u,x)\}.
$$
Note that $F^{-i}(X)= X_i= (x^s_i,x^u_i,x_i)\in \CC$, $x_i\in (0,1)$, for all $i\ge 0$.
As $F^{-1}$ uniformly expands the $s$-direction, by equation~\eqref{e.unstablemanifold}, there is $n>0$ with
\[
	F^{-n}(\Delta^s_\delta(X))\pitchfork W^u (Q^\ast , F).
\]	
Hence $\Delta^s_\delta(X)$ intersects $W^u(Q^\ast,F)$ in a point $X(\delta)=(x^s(\delta),x^u,x)$
with $x^s(\delta)\in [x^s-\delta,x^s+\delta]$.
For small $\varepsilon>0$ consider the $cu$-disk
$$
\Delta=
	\Delta_\varepsilon^{cu}(X(\delta))
	\eqdef \{x^s(\delta)\}
               \times [x^u-\varepsilon,x^u+\varepsilon]
               \times [x-\varepsilon,x+\varepsilon]
	\subset W^u (Q^\ast , F).
$$

\begin{claim*}
\emph{
	There are a disk $\Delta_0 \subset \Delta$ and a positive integer $n$ such that $F^j(\Delta_0) \subset \CC$
	for all $0\le j\le n$ and $F^n(\Delta_0)$ transversely intersects
	$W^s_{\rm loc}(Q^\ast ,F)$.
}	
\end{claim*}
The claim implies that $\Delta_0 \cap H_\CC(Q^\ast ,F)$. As $\delta,\,\varepsilon$ can be chosen arbitrarily small the
there are points of  $H_\CC(Q^\ast ,F)$
 arbitrarily close to $X$. Thus $X\in  H_\CC(Q^\ast ,F)$  implying the lemma.

\begin{proof}[Proof of the claim]
By expansion in the $u$-direction, there is some $m_0>0$ such that $F^{m_0}(\Delta)$ transversely intersects the subset $[0,1]\times (0,1) \times \{0\}$ of $W^s(P,F)$. Further, again by expansion in the $u$-direction, after new iterations we get $m_1>0$ such  that $F^{m_0+m_1} (\Delta)$ stretches over the full unstable direction, that is, it contains a fully stretching disk of the form
\[
	\Delta_1=\{x_1^s\} \times [0,1] \times I_1\quad\text{ for some }\quad
		x_1^s\in [0,1],\,\,I_1\subset (0,1)
\]		
(Step 1 in Figure~\ref{fig.steps}). Now, by the cycle configuration, the interval $I_1$ can be moved by the IFS generated by $f_0,f_1$ into the interval $(0,b)$ (Step 2) which for the disk means that there is $m_2>0$ with
\[
	F^{m_0+m_1+m_2} (\Delta)
	\supset \Delta_2
	=\{x_2^s\} \times [0,1] \times I_2, \quad
x_2^s\in [0,1],\,\, I_2\subset (0,b).
\]
Finally, considering further forward iterates we can assume that the central interval
$I_2$ of $\Delta_2$ satisfies $I_2\subset [f_0^{-2}(b),b]\subset
[f_1\circ f_0^N(b), f_0^N(b)]$ (Step 3).
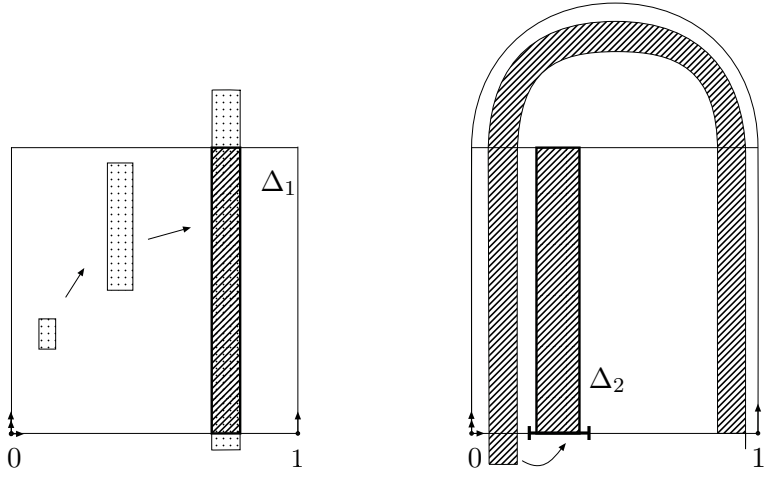
\begin{figure}
\begin{minipage}[c]{\linewidth}
\centering
\vspace{0.5cm}
\begin{overpic}[scale=.30,
  ]{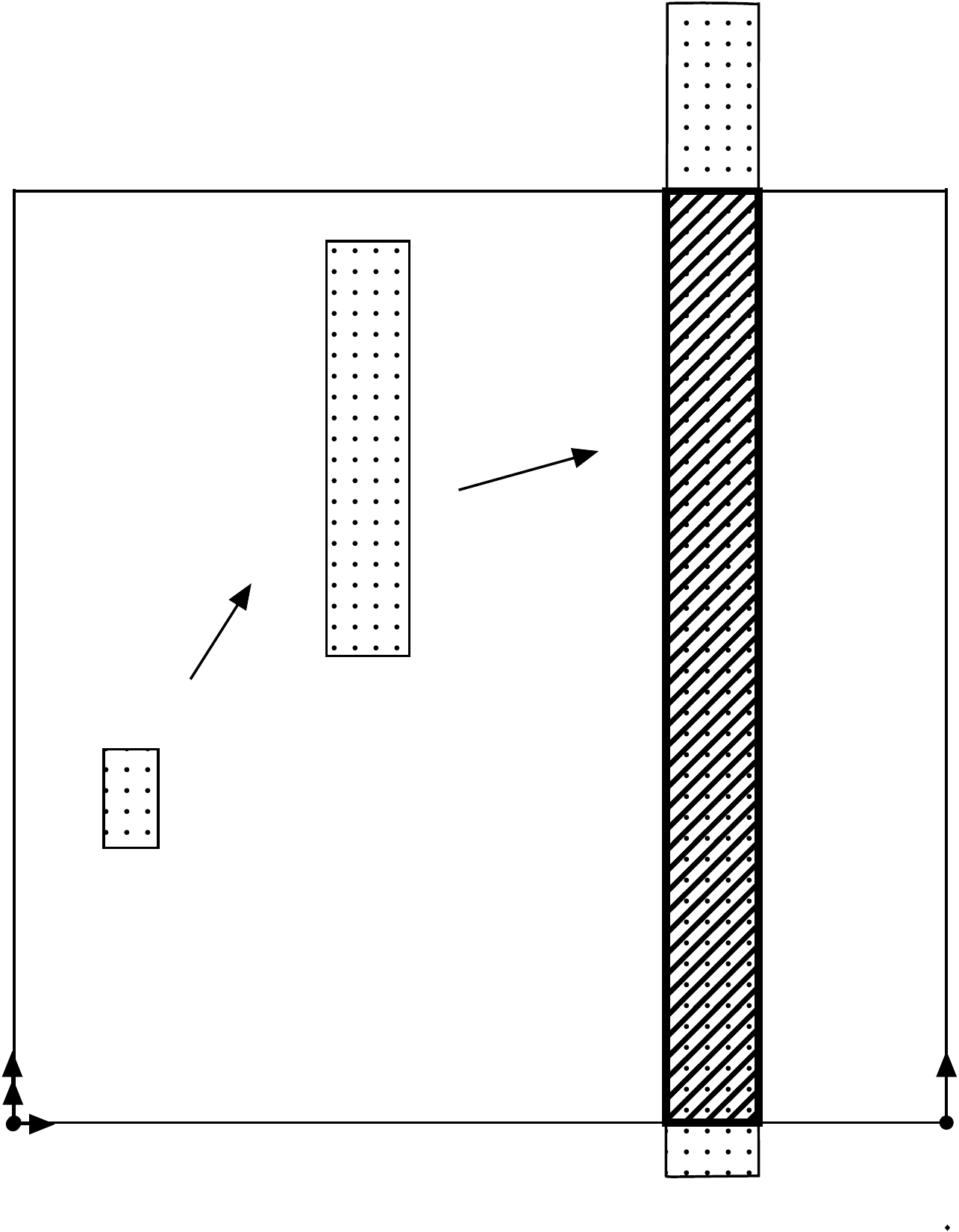}	
    	\put(67,73){$\Delta_1$}	
    	\put(75,0){\small $1$}	
    	\put(0,0){$0$}		
  \end{overpic}
  \hspace{2cm}
  \begin{overpic}[scale=.30,
  ]{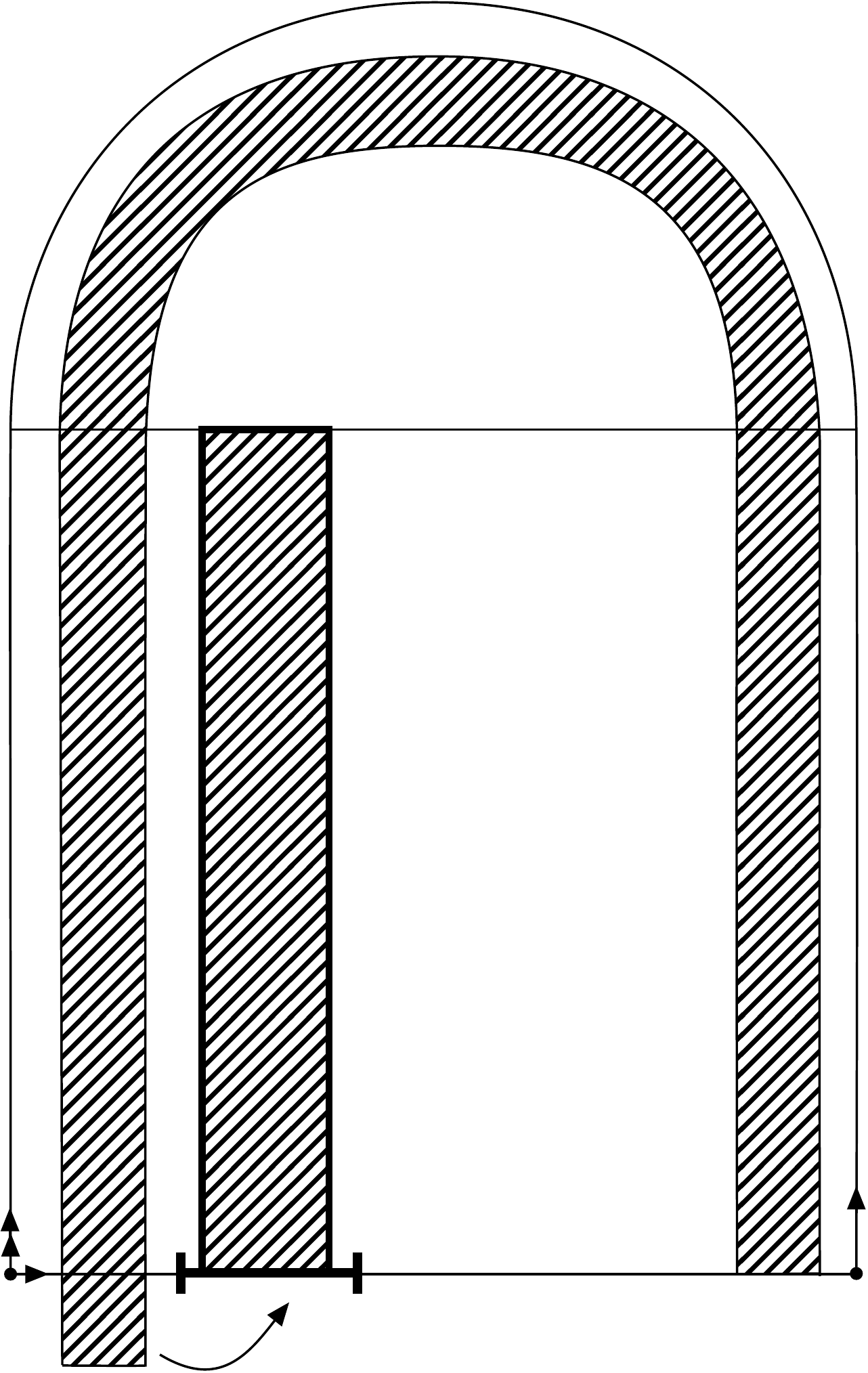}	
    	\put(26,17){$\Delta_2$}	
	\put(61,0){$1$}	
    	\put(0,0){$0$}
  \end{overpic}
\end{minipage}
\caption{Steps 1 and 3}
\label{fig.steps}
\end{figure}

In what follows our arguments only involve iterations of the fibre maps $f_0$, $\tilde f_{0}$ in $[f_1\circ f_0^N(b), f_0^N(b)]$ where these maps coincide (see (\~F02)).
Consider the expanding sequence  $\xi(I_2)=(\xi_0\ldots \xi_k)$ of $I_2$ and the disk
$$
	\Delta_{[\xi_0\ldots\xi_k]}\eqdef
	\Delta_2  \cap \big( \widehat\CC_{[\xi_0\ldots\xi_k]} \times I_2\big),
$$
where $\widehat\CC_{[\xi_0\ldots\,\xi_k]}$ is the sub-cube of $\widehat\CC$ defined in Section~\ref{sec:2.2}. By definition of the skew product dynamics, there is $\overline x^s\in [0,1]$ such that
\[
	F^{k+1}(\Delta_2)
	= \{\overline x^s\} \times [0,1]\times f_{[\xi_0\ldots\,\xi_k]}(I_2)
	\subset \{\overline x^s\} \times [0,1]\times [f_0^2(b),f_0^{-1}(b)],
\]
where the last inclusion follows from the definition of an expanding sequence, see Lemma~\ref{l.xJ}. Hence,~\eqref{e.manifolds} implies that $F^{k+1}(\Delta_{[\xi_0\ldots\xi_k]})$ transversely  meets $W^s(Q^\ast,F)$.
To complete the proof observe that the iterations of the part of the disc $\Delta$ which we consider remain completely in $\CC$.
\end{proof}
The proof of the lemma is now complete.
\end{proof}

\section{Proof of Theorem~\ref{t1}}\label{s:PROOF}

First recall that by Proposition~\ref{prop:locmax} the set $\Lambda$ is transitive. So what remains to prove is the existence of the rich phase transitions.

Birkhoff averages of $t\,\phi$ with respect to the shift map $\sigma$ naturally correspond to the central Lyapunov exponents of the skew product $F$.
Given $(\widehat x,0)\in\CC^3$, consider the \emph{upper central Lyapunov exponent} of $(\widehat x,0)$ with respect to $F$ given by
\[\begin{split}
		\overline\chi_c(\widehat x,0)&\eqdef
		\limsup_{n\to\infty}\frac 1 n \log\,\lVert dF^n|_{E^c(\widehat x,0)}\rVert\\
		&=\limsup_{n\to\infty}\frac1n\log\,\lvert (f_\xi^n)'(0)\rvert
		=\overline\chi(\varpi(\widehat x),0).
\end{split}\]

Consider the set $\Lambda_\bH$ defined in~\eqref{defLLLL}.
We split the set of \emph{ergodic} measures $\cM_\Lambda$ supported on $\Lambda$ into two disjoint sets: the set $\cM_\bH$ and $\widetilde \cM$ defined as the set of ergodic measures supported $\Lambda_{\bH}$ and its complement $\widetilde\cM$ in the $\cM_\Lambda$. For $\mu\in\cM_\Lambda$ we denote the \emph{central Lyapunov exponent} of $\mu$ by
\[
	\chi_c(\mu)\eqdef\int\varphi\,d\mu=\int\log\,\lVert dF|_{E^c}\rVert\,d\mu.
\]
The following result about the central exponents in $\cM$ is an immediate consequence of Remark~\ref{remremrem}.

\begin{lemma}\label{l:lorenzo}
	For $\mu$ ergodic, we have $\chi_c(\mu)\in[\log \beta_2, \log \beta_0]$ if, and only if, $\mu\in \cM_\bH$. Indeed, $\chi_c(\mu)\in[\alpha_\ell^-,\alpha_\ell^+]$ for some $\ell=1,\ldots,k$. Furthermore, any $\mu\in\widetilde\cM$ satisfies $\chi_c(\mu)\le\log \widetilde \beta$.
\end{lemma}

We consider the pressure functions
\[
	\widetilde P(t) \eqdef
	\sup_{\mu\in\widetilde\cM} \Big(h_\mu(F)+t\int\varphi\,d\mu\Big),
	\quad
	P_\bH(t) \eqdef
	\sup_{\mu\in\cM_\bH} \Big(h_\mu(F)+t\int\varphi\,d\mu\Big).
\]
Note that the topological pressure satisfies
\[
	P(t)\eqdef P(t\,\varphi)
		= \max\big\{ P_\bH(t),\widetilde P(t)\big\}.
\]
The following lemma and Proposition~\ref{p:main} immediately imply the theorem.

\begin{lemma}\label{lemmmmm}
	For $t\le t_1$ we have $P(t)=P_\bH(t)=\bP(t\,\phi)$.
\end{lemma}

\begin{proof}
	By construction, it is immediate that $P_\bH(t)=\bP(t\,\phi)$ for every $t$.

	To conclude the proof of the proposition, observe that by Lemma~\ref{l:lorenzo} we have
\[
	\widetilde P(t)	\le \log3-t\log\widetilde\beta
	\quad\text{ for every }t\le0.
\]	
Recall the definition of the parameter $t_0<0$ in~\eqref{e:important} that is the intersection of the straight lines $(t,\log 2-t\log\beta_0)$ and $(t,\log3-t\log\widetilde\beta)$. Hence,
\[
	\bP(t\,\phi)\ge \log3-t\log\widetilde\beta
	\quad\text{ for every }t\le t_0.
\]	
Therefore, for $t\le t_0$
\[
	P(t)
	= \max\big\{P_\bH(t),\widetilde P(t)\big\}=
	\bP(t).
\]
This proves the lemma.
\end{proof}

To obtain the equilibrium states it suffices to observe that $\Lambda_\bH\subset \Gamma\times\{0\}$ and to consider the conjugation $\varpi\colon\Gamma\to\Sigma_\bH$. Observe that, in fact,
\[
	P_{\sigma|\bH}(t\,\phi)
	= P_{\Phi|\Gamma}(t\,\phi\circ\varpi)
	= P_{F|\Lambda_\bH}(t\,\varphi),
\]
where $\phi$ was defined in~\eqref{eq2}.
For every $\ell=1,\ldots,k$, consider the equilibrium measures $\mu_\ell^\pm$ for $t_\ell\,\varphi$ with respect to $\sigma|_\bH$ provided by Proposition~\ref{p:main}. Define the measures $\nu_\ell^\pm\in \cM_\bH$ as follows: given $B=\widehat B\times\{0\}\subset \Lambda_\bH$ let
\[
	\nu_\ell^\pm(B)=\mu_\ell^\pm(\varpi(\widehat B)).
\]
By a standard argument, $\nu_\ell^\pm$ are equilibrium states for $t_\ell\,\varphi$ with respect to $F|\Lambda_\bH$. Hence, by Lemma~\ref{lemmmmm},  they are equilibrium states for $t_\ell\,\varphi$ with respect to $F|\Lambda$. This finishes the proof of the theorem.
\hfill$\square$

\bibliography{ref}

\end{document}